\documentclass[a4paper,12pt,twoside]{amsart}
\usepackage[english]{babel}
\usepackage{amssymb, amsmath, eucal}
\usepackage[all]{xy}
\usepackage{mathrsfs}
\usepackage[dvips]{graphics}

\newtheorem{The}{Theorem}
\newtheorem{Lem}[The]{Lemma}
\newtheorem{Prop}[The]{Proposition}
\newtheorem{Cor}[The]{Corollary}

\newcommand{\C}{\mathbb{C}}
\newcommand{\R}{\mathbb{R}}
\newcommand{\N}{\mathbb{N}}
\newcommand{\Z}{\mathbb{Z}}

\newcommand{\E}{\mathcal{E}}
\newcommand{\F}{\mathcal{F}}

\begin{document}
 \title[An integral theorem for plurisubharmonic functions]
 {An integral theorem for plurisubharmonic functions}

\setcounter{tocdepth}{1}

  \author{Hoang-Son Do} 
\address{Institute of Mathematics \\ Vietnam Academy of Science and Technology \\18
Hoang Quoc Viet \\Hanoi \\Vietnam}
\email{hoangson.do.vn@gmail.com , dhson@math.ac.vn}
 \date{\today}


\begin{abstract}
  In this paper, we prove an integral theorem 
  for Cegrell class $\mathcal{F}(f)$ and use this result to study the $\F$-equivalence relation.
\end{abstract}

\maketitle

\section*{Introduction}
 Let $\Omega\subset\C^n\, (n\geq 2)$ be a bounded hyperconvex domain. 
 Following \cite{Ceg98, Ceg04,  ACCP09}, we denote
\begin{center}
	$\E_0(\Omega)=\{u\in PSH^-(\Omega)\cap L^\infty (\Omega): \lim_{z\to\partial\Omega}u(z)=0, \int_\Omega (dd^cu)^n<\infty\},$\\ 
	$\F(\Omega)=\{u\in PSH^-(\Omega): \exists \, \{u_j\}\subset \E_0(\Omega), \; u_j\searrow u, \, \sup_{j}\int_{\Omega} (dd^cu_j)^n<\infty\}$ ,\\
	$\mathcal{E}(\Omega)
	=\{ u\in PSH^-(\Omega):\forall K\Subset\Omega, \exists u_K\in\mathcal{F}(\Omega)
	\mbox{ such that } u_K=u \mbox{ on } K\},$
\end{center}
and for every $f\in PSH^-(\Omega)$,
\begin{center}
	$\F(\Omega, f)=\{ u\in PSH^-(\Omega): \exists v\in\F\mbox{ such that } v+f\leq u\leq f \}.$
\end{center}
The class $\E$ is the largest subclass of $PSH^-(\Omega)$ on which the complex Monge-Amp\`ere operator 
is well-defined \cite{Ceg04, Blo06}. The class $\F$ is
the subclass of $\E$ containing those functions with
smallest maximal plurisubharmonic majorant identically zero and with finite total
Monge-Amp\`ere mass. If $f\in\E$ then $\F (f)\subset\E$.

Our main result is the following:
\begin{The}\label{main} Suppose  $(X, d, \mu)$ is a totally bounded metric probability space
	and $u, f: \Omega\times X\rightarrow [-\infty, 0]$ are measurable functions such that
	\begin{itemize}
		\item [(i)]  For every $a\in X$, $f(\cdot, a)\in\E$.
		\item [(ii)] For every $a\in X$, $u(\cdot, a)\in\F (f(., a))$ and
		\begin{center}
			$\int\limits_{U_a}(dd^cu(z, a))^n\leq (M(a))^n$,
		\end{center}
		where $M\in L^1(X)$ is given and $U_a=\Omega\cap\overline{\{z\in\Omega: u(z, a)<f(z,a)\}}$.
		\item[(iii)] The function $a\mapsto u(z, a)$ is upper semicontinuous in $X$ for  every $z\in\Omega$.
		\item [(iv)] The function $a\mapsto e^{f(z, a)}$ is lower semicontinuous in $X$ 
		for  every $z\in\Omega$.
		\item [(v)] The function $\tilde{f}(z):=\int\limits_Xf(z, a)d\mu(a)$ is not identically $-\infty$.
	\end{itemize}
	Then  $\tilde{u}(z):=\int\limits_Xu(z, a)d\mu(a)\in\F (\tilde{f})$. In particular, if
	$\tilde{f}\in\E$ then $\tilde{u}\in\E$ and $\int_{\Omega}(dd^c\tilde{u})^n<\infty$.
\end{The}
This result follows the plurisubharmonic version of
 \cite[Theorem 2.6.5]{Kli91} in the direction of focusing on the conservation of
the existence of Monge-Amp\`ere measures. We are not sure that the conditions (iii) and (iv) are necessary but
we need these conditions in our proof. Our method is as follows: 
 we solve the problem for the case $f\equiv 0$, then we use plurisubharmonic envelopes to reduce
 the problem to the case  $f\equiv 0$. In the first step, we consider a decreasing sequence of
 functions $u_j\in\F$ and prove that $\lim_{j\to\infty}u_j\in\F$. Then we use the condition (iii) to
 show that $u=\lim_{j\to\infty}u_j$. In the last step, we need the conditions (iii) and (iv) to reduce
 the problem to the case  $f\equiv 0$.

For $u_1, u_2\in\E(\Omega)$, we say that $u_1$ is $\F$-equivalent to $u_2$ if there exist $v_1, v_2\in\F$ 
such that $u_1+v_1\leq u_2$ and $u_2+v_2\leq u_1$. Observe that $u_1$ is $\F$-equivalent to $u_2$
iff $u_1, u_2\in\F (\max\{u_1, u_2\})$. The following result is an immediate corollary of Theorem \ref{main}:
\begin{Cor} Suppose  $(X, d, \mu)$ is a totally bounded metric probability space
	and $u, v: \Omega\times X\rightarrow [-\infty, 0]$ are measurable functions such that
	\begin{itemize}
		\item [(i)] For every $a\in X$, $u(\cdot, a), v(\cdot, a)\in\E(\Omega)$ and
		\begin{center}
			$\int\limits_{U_a}(dd^cu(z, a))^n
			+\int\limits_{V_a}(dd^cv(z, a))^n\leq (M(a))^n$,
		\end{center}
		where $M\in L^1(X)$ is given, $U_a=\Omega\cap\overline{\{z\in\Omega: u(z, a)<v(z,a)\}}$
		and $V_a=\Omega\cap\overline{\{z\in\Omega: v(z, a)<u(z,a)\}}$.
	 \item [(ii)] For every $a\in X$,  $u(\cdot, a)$ is $\F$-equivalent to $v(\cdot, a)$.
		\item [(iii)] The functions $a\mapsto e^{u(z, a)}$ and $a\mapsto e^{v(z, a)}$ are continuous in $X$ 
		for  every $z\in\Omega$.
	\end{itemize}
	Then  $\tilde{u}(z):=\int\limits_Xu(z, a)d\mu(a)\in\E$ iff
	 $\tilde{v}(z):=\int\limits_Xv(z, a)d\mu(a)\in\E$. Moreover, if $\tilde{u}, \tilde{v}\in\E$ then
	  $\tilde{u}$ is $\F$-equivalent to $\tilde{v}$.
\end{Cor}
In the next section, we  recall briefly some properties of the class $\F$ and
plurisubharmonic envelopes that  will be used to prove the main theorem.

\medskip
\noindent\textbf{Acknowledgement.} This paper was partially written while the author
visited Vietnam Institute for Advanced Study in Mathematics(VIASM). He would like to thank  
this institution for its hospitality and support.  
\section{Preliminaries}
\subsection{The class $\F$}
We recall some properties of the class $\F$. The reader can find more details in \cite{Ceg04, NP09}.

The following proposition is a corollary of \cite[Proposition 5.1]{Ceg04}:
\begin{Prop}\label{MA in F}
	Suppose $u\in\F (\Omega)$. If $u_j\in\E_0(\Omega)$
	decreases to $u$ as $j\rightarrow\infty$ then
	\begin{center}
		$\lim\limits_{j\to\infty}\int\limits_{\Omega}(dd^cu_j)^n
		=\int\limits_{\Omega}(dd^cu)^n.$
	\end{center}
	In particular, $\int\limits_{\Omega}(dd^cu)^n<\infty$.
\end{Prop}
\begin{Prop}\label{lem mix MA}\cite[Corollary 5.6]{Ceg04} Suppose $u_1,...,u_n\in\F (\Omega)$. Then
	\begin{center}
		$\int\limits_{\Omega}dd^cu_1\wedge...\wedge dd^cu_n
		\leq \left(\int\limits_{\Omega} (dd^cu_1)^n\right)^{1/n}...
		\left(\int\limits_{\Omega} (dd^cu_n)^n\right)^{1/n}.$
	\end{center}
\end{Prop}
\begin{Prop}\label{prop max F}
	a) If $u,v\in\F(\Omega)$ then $u+v\in\F(\Omega)$.\\
	b) If $u\in\F(\Omega)$ and $v\in PSH^-(\Omega)$ then $\max\{u, v\}\in\F(\Omega)$.
\end{Prop}
The part a) of Proposition \ref{prop max F} can be obtained by using 
 \cite[Lemma 5.4]{Ceg04}, Proposition \ref{MA in F} and the definition of the class $\F$.
 The part b) can be obtained by using the definition of the class $\F$ and the Bedford-Taylor
 Comparison Principle \cite{BT82}.
 
 By \cite[Theorem 3.7]{NP09}, we have:
\begin{Prop}\label{lem appr}
	Let $\Omega$ be a hyperconvex domain in $\C^n$ and $u\in PSH^-(\Omega)$. Assume that there are
	$u_j\in\F (\Omega)$, $j\in\N$, such that $u_j$ converges almost everywhere to
	$u$ as $j\rightarrow\infty$. If
	$\sup_{j>0}\int_{\Omega}(dd^cu_j)^n<\infty$ then $u\in\F(\Omega)$.
\end{Prop}
By \cite[Proposition 3.1]{NP09}, we have:
\begin{Prop}\label{cor NP compa}
	Let $u,v\in\F$ such that $u\leq v$ in $\Omega$. Then
	\begin{center}
	$\int\limits_\Omega(dd^cv)^n\leq\int\limits_\Omega(dd^cu)^n.$	
	\end{center}
\end{Prop}
\subsection{Plurisubharmonic envelopes}
Let $D \Subset \C^n$ be a bounded domain. If $u:D\rightarrow\R$ 
is a  bounded function then the plurisubharmonic envelope $P_{D}(u)$ of $u$ 
 in $D$ is defined by
 \begin{center}
 	$P_D(u)=(\sup\{v\in PSH(\Omega): v\leq u \})^*$,
 \end{center}
where $(\sup\limits_{v\in S}v(z))^*$ is the upper envelope of $\sup\limits_{v\in S}v(z)$.
\begin{Lem}\label{lem GLZ decrease seq}
	a) Let $u:D\rightarrow\R$ be a bounded function. Then $P_D(u)\leq u$ quasi everywhere, i.e., the set
	$\{z\in D: P_D(u)(z)>u(z) \}$ is pluripolar. Moreover,
	\begin{center}
		$P_D(u)=\sup \{v\in PSH(D): v\leq u$ quasi everywhere on  $ D \}.$
	\end{center}
b)	Let $u_j, u:D\rightarrow\R$ be bounded functions such that $u_j\searrow u$ as $j\rightarrow\infty$.
Then $P_D(u_j)$ decreases to $P_{D}(u)$.
\end{Lem}
The proof of Lemma \ref{lem GLZ decrease seq} is the same as the proof of the parts 1), 2) of
\cite[Proposition 2.2]{GLZ19}.
For every domain $W\Subset D$, we also consider
\begin{center}
	$P_{\overline{W}} (u) := (\sup\{v \in \mathrm{PSH}(W)  : \hat{v} \leq u\ \textrm{ on } \overline{W}\})^*,$
\end{center}
where $\hat{v}$ is the upper semicontinuous extension of $v$ to $\overline{W}$ defined by  
\begin{center}
	$\hat{v}(\xi):= \lim_{r\to 0^+} \sup_{B(\xi,r)\cap W} v, \ \forall\xi \in \partial W.$
\end{center}

  The following results are also proved in \cite{GLZ19}:
  \begin{Lem}
  	\label{lem GLZ exhaustion}\cite[Lemma 3.11]{GLZ19}
  	Let $(D_j)$ be an increasing sequence of relatively compact domains in $D$ such that $\cup D_j=D$. Assume that $u$ is a bounded lower semicontinuous function in $D$. Then $P_{\overline{D_j}}(u)$ decreases to $P_{D}(u)$. 
  \end{Lem}
 \begin{Lem}\label{lem: increasing sequence}\cite[Lemma 3.10]{GLZ19}
 	Let $(u_j)$ be an increasing sequence of continuous functions on $D$ which converges pointwise to a bounded function $u$. Let $W$ be a relatively compact  domain in $D$. Then $P_{W}(u_j)$ increases almost everywhere to $P_{\overline{W}}(u)$.  
 \end{Lem}
\begin{Prop} \label{Prop GLZ supersolloc}\cite[Theorem 3.9]{GLZ19}
	Let  $D \Subset \C^n$ be a bounded pseudoconvex domain. Assume that a bounded lower semi-continuous function $u$ is a viscosity supersolution (see \cite{EGZ} for the definition) of the equation
	\begin{equation}
	\label{eq: viscosity supersolution local 1}
	(dd^c u)^n =  fdV,
	\end{equation}
	in $D$. Then $P_{D}(u)$ is a pluripotential supersolution of \eqref{eq: viscosity supersolution local 1} in $D$. 
\end{Prop}

\section{Proof of the main result}
We first prove Theorem \ref{main} for the case $f=0$.
\begin{Prop}\label{prop f=0} Let $\Omega\subset\C^n$ be a bounded hyperconvex domain and $(X, d, \mu)$ be a totally bounded metric probability space. 
	Let $u: \Omega\times X\rightarrow [-\infty, 0]$ such that
	\begin{itemize}
		\item [(i)] For every $a\in X$, $u(\cdot, a)\in\F (\Omega)$ and
		\begin{center}
			$\int\limits_{\Omega}(dd^cu(z, a))^n\leq (M(a))^n$,
		\end{center}
	where $M\in L^1(X)$ is given.
	\item[(ii)] For every $z\in\Omega$, the function $u(z, \cdot)$ is upper semicontinuous in $X$.
	\end{itemize}
Then  $\tilde{u}(z):=\int\limits_Xu(z, a)d\mu(a)\in\F (\Omega)$. Moreover
\begin{center}
	$\int\limits_{\Omega}(dd^c\tilde{u})^n\leq (\int\limits_X M(a)d\mu(a))^n$.
\end{center}
\end{Prop}
\begin{proof}
	We will show that there exists a sequence of functions $\tilde{u}_j\in\F(\Omega)$ such that
	 $\tilde{u}_j\searrow\tilde{u}$ as $j\rightarrow\infty$ 
	\begin{center}
		$\sup\limits_{j\in\Z^+}\int\limits_{\Omega}(dd^c\tilde{u}_j)^n\leq M(a),$
	\end{center}
for every $a\in X$.

	Since $X$ is totally bounded, there exists a finite cover $\{X_k\}_{k=1}^{m_1}$ of $X$ 
	such that the diameter of each $X_k$ is at most $1/2$. Denote
	\begin{center}
$U_{1,1}=X_1, U_{1, 2}=X_2\setminus X_1,...,
 U_{1, m_1}=X_{m_1}\setminus(\cup_{l=1}^{m_1-1}X_l).$
	\end{center}
	Then $\{U_{1, k}\}_{k=1}^{m_1}$ is a finite cover of $X$ such that
	\begin{itemize}
		\item $U_{1, k}\cap U_{1, l}=\emptyset$ if $k\neq l$;
		\item $diam(U_{1, k})\leq 1/2$ for all $k=1,...,m_1$;
		\item $U_{1, k}$ is totally bounded for all $k=1,...,m_1.$
	\end{itemize}
	 By using induction, for every $j\in\Z^+$, we can find a finite cover $\{U_{j, k}\}_{k=1}^{m_j}$ of $X$
	 such that
	 \begin{itemize}
	 	\item For every $1\leq k\leq m_{j+1}$, there exists $1\leq l\leq m_j$ such that
	 	$U_{j+1, k}\subset U_{j, l}$;
	 	\item $U_{j, k}\cap U_{j, l}=\emptyset$ if $k\neq l$;
	 	\item $diam(U_{j, k})\leq 2^{-j}$ for all $k=1,...,m_1$.
	 \end{itemize}
	 
	 For every $j\in \Z^+$, we define
	 \begin{center}
	 	$u_j(z)=\sum\limits_{k=1}^{m_j}\mu (U_{j, k})\sup\limits_{a\in U_{j,k}}u(z, a)\qquad$ and $\qquad\tilde{u}_j=(u_j)^*$.
	 \end{center}
 Then $\tilde{u}_j\in\F (\Omega)$. 
 Let $a_{j,k}$ be an arbitrary element of $U_{j,k}$ for $j\in\Z^+$ and $k=1,..., m_j$.
  By using Proposition \ref{cor NP compa} for 
 $\tilde{u}_j$ and $\sum\limits_{k=1}^{m_j}\mu (U_{j, k})u(z, a_{j, k})$
 and by applying Proposition \ref{lem mix MA}, we have
 \begin{flushleft}
 	$\int\limits_{\Omega}(dd^c\tilde{u}_j)^n 
 	\leq \int\limits_{\Omega}(dd^c(\sum\limits_{k=1}^{m_j}\mu (U_{j, k})u(z, a_{j, k})))^n$\\
 	$=\sum\limits_{k_1+...+k_{m_j}=n}
 	\dfrac{n!}{k_1!...k_{m_j}!}\left(\prod\limits_{l=1}^{m_j}\mu (U_{j, l})^{k_l}\right)
 	\int\limits_{\Omega}(dd^cu (z, a_{j, 1}))^{k_1}\wedge...\wedge (dd^cu(z, a_{j, m_j}))^{k_{m_j}}$\\
 	$\leq \sum\limits_{k_1+...+k_{m_j}=n}
 	\dfrac{n!}{k_1!...k_{m_j}!}\left(\prod_{l=1}^{m_j}\mu (U_{j, l})^{k_l}\right)
 	\prod_{l=1}^{m_j}(\int\limits_{\Omega}(dd^cu (z, a_{j, l}))^n)^{k_l/n}$\\
 		$\leq \sum\limits_{k_1+...+k_{m_j}=n}
 	\dfrac{n!}{k_1!...k_{m_j}!}\left(\prod_{l=1}^{m_j}\mu (U_{j, l})^{k_l}\right)
 	\prod_{l=1}^{m_j}M(a_{j, l})^{k_l}$\\
 		$\leq \sum\limits_{k_1+...+k_{m_j}=n}
 	\dfrac{n!}{k_1!...k_{m_j}!}\left(\prod_{l=1}^{m_j}\mu (U_{j, l})^{k_l}\right)
 	\prod_{l=1}^{m_j}(\int\limits_{\Omega}(dd^cu (z, a_{j, l}))^n)^{k_l/n}$\\
 	$=\sum\limits_{k_1+...+k_{m_j}=n}
 	\dfrac{n!}{k_1!...k_{m_j}!}\prod_{l=1}^{m_j}\mu (U_{j, l})M(a_{j, l}))^{k_l}$\\
 	$=(\mu (U_{j, 1})M(a_{j, 1})+...+\mu (U_{j, k_{m_j}}M(a_{j, k_{m_j}})))^n$,
 \end{flushleft}
for all  $j\in\Z^+$. Since $a_{j,k}$ is arbitrary for every $j, k$, we have
\begin{equation}
	\int\limits_{\Omega}(dd^c\tilde{u}_j)^n \leq
	 \left(\sum\limits_{k=1}^{m_j}\mu(U_{j,k})\inf\limits_{U_{j,k}}M(a)\right)^n
	 \leq \left(\int\limits_X M(a)d\mu(a)\right)^n,
\end{equation}
for all  $j\in\Z^+$.

We will show that $\tilde{u}_j$ is decreasing to $\tilde{u}$ and use Proposition \ref{lem appr} to 
prove that $\tilde{u}\in\F(\Omega)$.

 For every $z\in\Omega, a\in X$ and $j\in\Z^+$, we define
 \begin{center}
 	$\phi_j (z, a)=\sum\limits_{k=1}^{m_j}\chi_{U_{j, k}}(a)\sup\limits_{a\in U_{j,k}}u(z, a)=
 	\sup\limits_{\xi\in U_{j,k(j, a)}}u(z, \xi)$,
 \end{center}
where $\chi_{U_{j, k}}$ is the  characteristic function of $U_{j, k}$ and
  $k(j, a)$ is given by $a\in U_{j, k(j, a)}$. Then, we have
  \begin{equation}\label{eq uj>u proofmain2}
   u_j(z)=\int\limits_X\phi_j(z, a)d\mu(a)\geq \int\limits_Xu(z, a)d\mu(a)=\tilde{u}(z),
  \end{equation}
 for every $z\in\Omega$ and $j\in\Z^+$.
 
 Note that $a\in U_{j+1, k(j+1, a)}\cap U_{j, k(j, a)}\neq\emptyset$. Then, by the construction
  of the sets $U_{j, k}$, we have $U_{j+1, k(j+1, a)}\subset U_{j, k(j, a)}$. Hence
  \begin{equation}\label{eq uj>uj+1 proofmain2}
   u_j(z)=\int\limits_X\phi_j(z, a)d\mu(a)\geq \int\limits_X\phi_{j+1}(z, a)d\mu(a)=u_{j+1}(z),
  \end{equation}
for every $z\in$ and $j\in\Z^+$.

By the semicontinuity of $u(z, \cdot)$, we have, 
\begin{equation}\label{eq u and phij proofmain2}
	u(z, a)\geq\lim\limits_{j\to\infty}(\sup\{u(z, \xi): |\xi-a|\leq 2^{-j}\})
	\geq\lim\limits_{j\to\infty} \phi_j (z, a),
\end{equation}
for every  $z\in\Omega$ and $a\in X$. By integrating the sides of
\eqref{eq u and phij proofmain2} with respect to $a$ and using Fatou's lemma, we get
\begin{equation}\label{eq uj<u proofmain2}
\tilde{u}(z)\geq\lim\limits_{j\to\infty}u_j(z).
\end{equation}
Combining \eqref{eq uj>u proofmain2}, \eqref{eq uj>uj+1 proofmain2} and \eqref{eq uj<u proofmain2}, we get
that $u_j$ is decreasing to $\tilde{u}$ as $j\rightarrow\infty$. Note that $u_j=\tilde{u}_j$ almost everywhere 
\cite[Proposition 2.6.2]{Kli91}, and then $\lim\limits_{j\to\infty}\tilde{u}_j=\tilde{u}$ almost everywhere.
Since $\lim\limits_{j\to\infty}\tilde{u}_j$ is either plurisubharmonic or identically $-\infty$, we have 
$\lim\limits_{j\to\infty}\tilde{u}_j=\tilde{u}$ everywhere.
Therefore, $\tilde{u}_j$ is decreasing to $\tilde{u}$ as $j\rightarrow\infty$.

 By Proposition \ref{lem appr}, $\max\{\tilde{u}, -k\}\in\F(\Omega)$ for $k>0$ and it implies that $\tilde{u}$ is not identically $-\infty$.
 Then, by using Proposition \ref{lem appr} for $\tilde{u}$, 
 we get that $\tilde{u}\in\F(\Omega)$. Moreover, since the sequence $\tilde{u}_j$ is decreasing, we have
 \begin{center}
 	 $\int\limits_{\Omega}(dd^c\tilde{u})^n\leq\liminf\limits_{j\to\infty}\int\limits_{\Omega}(dd^c\tilde{u}_j)^n\leq (\int\limits_X M(a)d\mu(a))^n$.
 \end{center}
\end{proof}
In order to prove Theorem \ref{main}, we need the following proposition:
\begin{Prop}\label{prop supersol}
	Let $\varphi\in\E(\Omega)$ and $u\in\F(\varphi)$. Define
	\begin{center}
		$\phi (u):=(\sup\{v\in PSH^-(\Omega): v+\varphi\leq u \})^*$.
	\end{center}
Then $\phi(u)\in\F$, $\phi(u)+\varphi\leq u$ and $(dd^c\phi(u))^n\leq \chi_U(dd^cu)^n$, where
$U=\Omega\cap\overline{\{u<\varphi\}}$.
\end{Prop}
We proceed through some lemmas.
\begin{Lem}\label{lem1 supersol}
	Let $u\in C(\overline{\Omega})\cap PSH(\Omega)$ and $v\in L^{\infty}(\Omega)\cap PSH(\Omega)$.
	Then, for every relatively compact pseudoconvex domain $W$ in $\Omega$, 
	$P_{\overline{W}}(u-v)\in L^{\infty}(W)\cap PSH(W)$
	 and $(dd^cP_{\overline{W}}(u-v))^n\leq (dd^c u)^n$ on $W$.
\end{Lem}
\begin{proof}
	Since $u|_W-\sup_W v\leq P_{\overline{W}}(u-v)\leq u|_W-\inf_W v$, we have 
	$P_{\overline{W}}(u-v)\in L^{\infty}(W)$. It remains to show that
	 $(dd^cP_{\overline{W}}(u-v))^n\leq (dd^c u)^n$ on $W$.
	 
	Let $u_j, v_j$ be sequences of smooth plurisubharmonic functions on a neighborhood of $\overline{W}$
	such that $u_j\searrow u$ and $v_j\searrow v$ as $j\rightarrow\infty$. Then, for every $j, k\geq 1$, the
	function $u_j-v_k$ is a viscosity supersolution to the equation
	\begin{equation}\label{eq1 lem supersol}
	(dd^c w)^n=(dd^cu_j)^n,
	\end{equation}
	on $W$.  It follows from Proposition \ref{Prop GLZ supersolloc} that the function
	$P_{W}(u_j-v_k)\in L^{\infty}(W)\cap PSH(W)$ satisfies
	\begin{equation}\label{eq2 lem supersol}
		(dd^cP_{W}(u_j-v_k))^n\leq (dd^c u_j)^n, 
	\end{equation}
on $W$ in the pluripotential sense. Moreover, by Lemma \ref{lem GLZ decrease seq}, we have
\begin{equation}\label{eq3 lem supersol}
P_{W}(u_j-v_k)\searrow P_{W}(u-v_k),
\end{equation}
as $j\rightarrow\infty$.
Combining \eqref{eq2 lem supersol} and \eqref{eq3 lem supersol}, we have
	\begin{equation}\label{eq4 lem supersol}
(dd^cP_{W}(u-v_k))^n\leq (dd^c u)^n.
\end{equation}

 By Lemma \ref{lem: increasing sequence}, we also have $P_{W}(u-v_k)\nearrow P_{\overline{W}}(u-v)$ almost everywhere
 as $k\rightarrow\infty$. Therefore, by \eqref{eq4 lem supersol}, we have
\begin{center}
$(dd^cP_{\overline{W}}(u-v))^n\leq (dd^c u)^n$.
\end{center}

\end{proof}
\begin{Lem}\label{lem2 supersol}
	Let $u\in C(\overline{\Omega})\cap PSH(\Omega)$ and $v\in L^{\infty}(\Omega)\cap PSH(\Omega)$.
	Then $P_{\Omega}(u-v)\in L^{\infty}(\Omega)\cap PSH(\Omega)$ and $(dd^cP_{\Omega}(u-v))^n\leq (dd^c u)^n$.
\end{Lem}
\begin{proof}
	Since $u-\sup_{\Omega} v\leq P_{\Omega}(u-v)\leq u-\inf_{\Omega}v$, we have  
		$P_{\Omega}(u-v)\in L^{\infty}(\Omega)\cap PSH(\Omega)$.
	Let $(\Omega_j)$ be an increasing sequence of relatively compact pseudoconvex domains in $\Omega$ such that 
	$\cup_{j\in\Z^+}\Omega_j=\Omega$. It follows from Lemma \ref{lem1 supersol} that 
	\begin{center}
	$(dd^cP_{\overline{\Omega_j}}(u-v))^n\leq (dd^c u)^n$,
	\end{center}
	 on $\Omega_j$ for every $j\in\Z^+$.
	 Moreover, by Lemma \ref{lem GLZ exhaustion}, we have $P_{\overline{\Omega_j}}(u-v)$ decreases to $P_{\Omega}(u-v)$.
	 	Hence, we have
	 	\begin{center}
	 			$(dd^cP_{\Omega}(u-v))^n\leq (dd^c u)^n$, on $\Omega$.
	 	\end{center}
\end{proof}
\begin{proof}[Proof of Proposition \ref{prop supersol}]
	By the assumption, there exists $v\in\F$ such that $v+\varphi\leq u\leq\varphi$. Then $v\leq\phi (u)\leq 0$. It follows from
	Proposition \ref{prop max F} that $\phi (u)\in\F$.
	By the definition of $\phi(u)$, we have $\phi(u)+\varphi\leq u$ almost everywhere. Therefore, by the subharmonicity of
	$\phi(u)+\varphi$ and $u$, we have  $\phi(u)+\varphi\leq u$.
	 It remains to show that $(dd^c\phi(u))^n\leq (dd^cu)^n$.
	
	Since $u\in PSH^-(\Omega)$, it follows from \cite[Theorem 2.1]{Ceg04} that there exists a sequence of functions
	$u_j\in\E_0(\Omega)\cap C(\overline{\Omega})$ such that $u_j\searrow u$ as $j\rightarrow\infty$. For every $j\in\Z^+$,
	we denote
	\begin{center}
		$w_j=u_j-\max\{\varphi, u_j\}$.
	\end{center}
We have
\begin{center}
	$ w_j=u_j-\dfrac{\varphi+u_j+|\varphi-u_j|}{2}=\dfrac{u_j-\varphi-|\varphi-u_j|}{2}=\min\{-\varphi+u_j, 0 \}$.
\end{center}
Then
\begin{equation}\label{eq1 prop supersol}
\phi (u)\leq w_{j+1}\leq w_j\leq 0,
\end{equation}
for every $j\in\Z^+$. Hence
\begin{equation}\label{eq2 prop supersol}
\phi(u)\leq P_{\Omega}(w_{j+1})\leq P_{\Omega}(w_j)\leq 0,
\end{equation}
for every $j\in\Z^+$. In particular, $ P_{\Omega}(w_j)\in\F(\Omega)$ for every $j$.

Since $P_{\Omega}(w_j)+\max\{\varphi, u_j\}$ and $u_j$ are plurisubharmonic and  $P_{\Omega}(w_j)+\max\{\varphi, u_j\}\leq w_j$
almost everywhere,
we have $P_{\Omega}(w_j)+\max\{\varphi, u_j\}\leq u_j$ for all $z\in\Omega$. Letting $j\rightarrow\infty$, we get
\begin{equation}\label{eq3 prop supersol}
\lim\limits_{j\to\infty}P_{\Omega}(w_j)+\varphi\leq u.
\end{equation}
Combining \eqref{eq2 prop supersol} and \eqref{eq3 prop supersol}, we get
$P_{\Omega}(w_j)\searrow\phi(u)$
as $j\rightarrow\infty$.
Moreover, by Lemma \ref{lem2 supersol}, we have $(dd^cP_{\Omega}(w_j))^n\leq (dd^cu_j)^n$. Therefore, by letting
$j\rightarrow\infty$, we obtain $(dd^c\phi(u))^n\leq (dd^cu)^n$. Observe that $\phi(u)$ is
 maximal plurisubharmonic (see \cite{Sad81}, \cite{Kli91} for the definition)
  on $\Omega\setminus \overline{\{u<\phi\}}=Int\{u=\phi\}$. Then, we have
 $(dd^c\phi(u))^n=0$ on $\Omega\setminus \overline{\{u<\phi\}}$. Thus
 $(dd^c\phi(u))^n\leq \chi_U(dd^cu)^n$.
\end{proof}
\begin{proof}[Proof of Theorem \ref{main}]
As in the proposition \ref{prop supersol}, for all $a\in X$, we define
\begin{center}
	$\begin{array}{rl}
\phi (u)(\cdot, a)&:=(\sup\{v\in PSH^-(\Omega): v+f\leq u(\cdot, a)\})^*\\
&\: =\sup\{v\in\F(\Omega): v+\varphi\leq u(\cdot, a)\}.
	\end{array}$
\end{center}
 For every $a\in X$, we have
\begin{center}
	$\begin{array}{ll}
	u(z, a)\geq\limsup\limits_{\xi\to a}u(z, \xi)&\geq 
	\limsup\limits_{\xi\to a}(\phi(u)(z, \xi)+f(z, \xi))\\
	& \geq
	\limsup\limits_{\xi\to a}\phi(u)(z, \xi)+\liminf\limits_{\xi\to a}f(z, \xi)\\
	&\geq \limsup\limits_{\xi\to a}\phi(u)(z, \xi)+f(z, a).
	\end{array}$
\end{center}
 Hence
\begin{center}
	$\phi (u)(z, a)\geq \limsup\limits_{\xi\to a}\phi(u)(z, \xi)$.
\end{center}
Moreover, by Proposition \ref{prop supersol}, we have
\begin{center}
	$\int\limits_{\Omega}(dd^c\phi (u)(z, a))^n\leq \int\limits_{U_a}(dd^cu(z, a))^n\leq (M(a))^n.$
\end{center}
Hence, the function $\phi (u)$ satisfies the assumptions of 
Proposition \ref{prop f=0}. Then $\widetilde{\phi (u)}:=\int_X\phi (u)d\mu(a)\in\F(\Omega)$.
 In the other hand, we have
\begin{center}
	$\widetilde{\phi (u)}+\tilde{f}\leq\tilde{u}\leq\tilde{f}.$
\end{center}
Thus $\tilde{u}\in\F(\tilde{f})$.
\end{proof}

\end{document}